\documentclass[12pt]{article}

\usepackage{amsthm}
\usepackage{amsmath}
\usepackage{silence}
\usepackage[bookmarks=true]{hyperref}

\newlength{\abstractwidth}
\renewcommand{\thanks}[1]{\footnote{#1}} 

\newcommand{\be}{\begin{equation}}
\newcommand{\ee}{\end{equation}}

\def\Re{{\rm Re\,}}

\def\log{\,{\rm log}\,}
\def\ddb{{\partial\bar\partial}}
\def\Tr{{\rm Tr\,}}
\def\det{{\rm det\,}}
\def\R{{\bf R}}

\newtheorem{lemma}{Lemma}
\newtheorem{theorem}{Theorem}

\newtheorem{corollary}{Corollary}

\begin{document}

\baselineskip=15pt

\begin{center}
{\Large \bf Parabolic complex Monge-Amp\`ere equations on compact Hermitian manifolds}
\end{center}

\begin{center}
Kevin Smith
\end{center}

\renewcommand{\abstractname}{\vspace{-\baselineskip}}
\begin{abstract}
\noindent \textbf{Abstract.} We prove the long-time existence and convergence of solutions to a general class of parabolic equations, not necessarily concave in the Hessian of the unknown function, on a compact Hermitian manifold. The limiting function is identified as the solution of an elliptic complex Monge-Amp\`ere equation.
\end{abstract}

\section{Introduction}\label{SectionIntroduction}

Since the celebrated solution of the Calabi conjecture by solving the complex Monge-Amp\`ere equation on a compact K\"ahler manifold \cite{Y}, fully non-linear partial differential equations have played an important role in complex geometry.

Motivated by unified string theories and particularly mirror symmetry, there has recently been interest in extending the theory of such equations beyond the K\"ahler setting to more general classes of Hermitian manifolds \cite{S, Pi, Ph, FHP, FGV, FY1, FY2, PPZ2, PPZ3, ST, BV, FFR}.

Additionally, new parabolic equations have emerged which are not concave in the eigenvalues of the Hessian of the unknown function \cite{PPZ3, PPZ4, CJY, Co}. This concavity property was an important requirement for fully non-linear equations as originally formulated in \cite{CNS1, CNS2, CNS3}.

The main theorem is:

\begin{theorem}\label{Theorem1}
Let $(X,\chi)$ be an $n$-dimensional compact Hermitian manifold, $f : X\rightarrow\R$ a smooth function, $F : (0,\infty)\rightarrow\R$ a strictly increasing smooth function, and $u_0 : X\rightarrow\R$ a smooth function satisfying $\chi + i\ddb u_0 > 0$. Then the equation
\be\label{ParabolicComplexMongeAmpereEquation}
\partial_t u = F(e^{-f}\det h)
\ee
admits a smooth solution which exists for all time and satisfies $\chi + i\ddb u > 0$ as well as $u(0) = u_0$, where $h^j{}_\ell = \delta^j{}_\ell + \chi^{j\bar k}u_{\bar k\ell}$. Furthermore the normalized solution
\be
\varphi = u - \frac{1}{V}\int_Xu\,\chi^n
\ee
converges in $C^\infty$ to the unique smooth function $\varphi_\infty$ solving the complex Monge-Amp\`ere equation
\be
(\chi + i\ddb\varphi_\infty)^n = ce^f\chi^n, \ \ \ \ \ c = \frac{\int_X(\chi + i\ddb\varphi_\infty)^n}{\int_Xe^f\chi^n}
\ee
satisfying $\chi + i\ddb\varphi_\infty > 0$ and $\int_X\varphi_\infty\chi^n = 0$.
\end{theorem}

\noindent Equation (\ref{ParabolicComplexMongeAmpereEquation}) has been studied for various concave choices of $F$ in the K\"ahler and Hermitian settings \cite{Ca, CHT, TW3}. Because we do not assume $F$ to be concave, the general theory of parabolic equations developed in \cite{CNS1, CNS2, CNS3} or \cite{PT} can not be applied.

Theorem \ref{Theorem1} is a generalization of the theorem of \cite{PZ}, which was proved in the K\"ahler setting. To prove Theorem \ref{Theorem1}, it suffices to prove the following $C^2$ estimate.

\begin{theorem}\label{Theorem2}
With the setup of Theorem \ref{Theorem1}, suppose that $u$ solves (\ref{ParabolicComplexMongeAmpereEquation}) on $X\times[0,T]$ for some $0 < T \leq \infty$ with $u(0) = u_0$. Then there exists a constant $C > 0$ depending only on $X$, $\chi$, $f$, $F$, $u_0$ so that
\be\label{C2Estimate}
\Tr h\leq Ce^{C(\varphi - \inf\varphi)}.
\ee
\end{theorem}

\noindent Here and throughout this paper, we are writing $\inf\varphi = \inf_{X\times[0,T]}\varphi$.

The reason that Theorem \ref{Theorem2} implies Theorem \ref{Theorem1} is that the $C^2$ estimate is essentially the only place in \cite{PZ} that the K\"ahler property is used, and so the rest of the proof carries over to the Hermitian setting. We remark that the elliptic $C^0$ estimate from \cite{TW2} may be used as an alternative to the $C^0$ estimate from \cite{PZ}.

In order to streamline the proof of Theorem \ref{Theorem2}, we make use of Corollary \ref{Corollary1} below. Corollary \ref{Corollary1} follows from the following generalization of an inequality of Aubin and Yau, which we prove in section \ref{SectionAubinYauInequality} using a special coordinate system of \cite{GL}.

\begin{theorem}\label{Theorem3}
If $\chi$ and $\omega$ are two Hermitian metrics on a compact complex manifold, then
\begin{align}
& \frac{g^{j\bar k}\partial_j\Tr h\partial_{\bar k}\Tr h}{\Tr h} - \chi^{p\bar q}\Tr(h^{-1}\nabla_ph h^{-1} \nabla_{\bar q}h) - \Re\chi^{p\bar q}\chi_{\bar sr}g^{j\bar k}\nabla_ph^r{}_j\overline{T^s_{qk}} - \Re g^{j\bar k}\nabla_{\bar k}h^p{}_rT^r_{pj} \nonumber \\
& \leq - 2\Re g^{j\bar k}\frac{\partial_j\Tr h}{\Tr h}\overline{T_k} - \frac{1}{\Tr h}g^{j\bar k}T_j\overline{T_k} + \chi^{p\bar q}g^{j\bar k}\chi_{\bar sr}T^r_{pj}\overline{T^s_{qk}} \nonumber \\
& \ \ \ + 2\Re g^{j\bar k}\frac{\partial_j\Tr h}{\Tr h}\overline{\alpha_k} + \frac{1}{\Tr h}2\Re g^{j\bar k}T_j\overline{\alpha_k} - \Re\chi^{p\bar q}g^{j\bar k}T^r_{pj}\overline{\alpha_{\bar rqk}} - \frac{1}{\Tr h}g^{j\bar k}\alpha_j\overline{\alpha_k}. \label{AubinYauInequality}
\end{align}
\end{theorem}

\noindent Here, we define the relative torsion $\alpha_{\bar qpj} = \partial_j(g_{\bar qp} - \chi_{\bar qp}) - \partial_p(g_{\bar qj} - \chi_{\bar qj})$ (that is, $i\alpha = \partial\omega - \partial\chi$) as well as $\alpha_j = \chi^{p\bar q}\alpha_{\bar qpj}$. For the rest of the notation in (\ref{AubinYauInequality}) we refer the reader to section \ref{SectionSetupAndNotation}.

We note that in the proof of Theorem \ref{Theorem3} and in the following corollary, we rewrite (\ref{AubinYauInequality}), using the identity
\begin{align}
\chi^{p\bar q}\Tr(h^{-1}\nabla_ph h^{-1} \nabla_{\bar q}h) + \Re\chi^{p\bar q}\chi_{\bar sr}g^{j\bar k}\nabla_ph^r{}_j\overline{T^s_{qk}} + \Re g^{j\bar k}\nabla_{\bar k}h^p{}_rT^r_{pj} \\
= \chi^{p\bar q}g^{j\bar k}g^{r\bar s}\nabla_pg_{\bar sj}\nabla_{\bar q}g_{\bar kr} + \Re\chi^{p\bar q}g^{j\bar k}\nabla_pg_{\bar sj}\overline{T^s_{qk}} + \Re\chi^{p\bar q}g^{j\bar k}\nabla_{\bar k}g_{\bar qr}T^r_{pj}. \nonumber
\end{align}
\begin{corollary}\label{Corollary1}
If $\partial\omega = \partial\chi$ (in particular if $\omega - \chi$ is $i\ddb$-exact) then
\begin{align}
& \frac{g^{j\bar k}\partial_j\Tr h\partial_{\bar k}\Tr h}{\Tr h} - \chi^{p\bar q}g^{j\bar k}g^{r\bar s}\nabla_pg_{\bar sj}\nabla_{\bar q}g_{\bar kr} - \Re\chi^{p\bar q}g^{j\bar k}\nabla_pg_{\bar sj}\overline{T^s_{qk}} - \Re\chi^{p\bar q}g^{j\bar k}\nabla_{\bar k}g_{\bar qr}T^r_{pj} \nonumber \\
& \leq - 2\Re g^{j\bar k}\frac{\partial_j\Tr h}{\Tr h}\overline{T_k} - \frac{1}{\Tr h}g^{j\bar k}T_j\overline{T_k} + \chi^{p\bar q}g^{j\bar k}\chi_{\bar sr}T^r_{pj}\overline{T^s_{qk}}. \label{CorollaryAubinYau}
\end{align}
\end{corollary}

\noindent To prove Theorem \ref{Theorem2} we use the test function
\be\label{TestFunction}
G = \log\Tr h - A\varphi + \frac{1}{\varphi - \inf\varphi + 1} + \frac{B}{2}F^2
\ee
and combine the techniques of \cite{PZ} and \cite{TW3}. We briefly discuss the motivation for including each term of (\ref{TestFunction}).

The first two terms of (\ref{TestFunction}) were used to obtain (\ref{C2Estimate}) for the complex Monge-Amp\`ere equation on a compact K\"ahler manifold \cite{A, Y}. From this Yau obtained a priori estimates up to order two, thus solving the Calabi conjecture and proving what is now called Yau's theorem.

The third term was introduced by \cite{PS} in a different setting, to obtain a $C^1$ estimate in the absence of a $C^0$ estimate. Subsequently this term was used in \cite{TW3} to obtain a $C^2$ estimate for the complex Monge-Amp\`ere equation on a compact Hermitian manifold.

Finally, the fourth term was introduced in \cite{PPZ4} to provide an alternative proof of Yau's theorem using the Anomaly flow \cite{PPZ1}, and the result was generalized in \cite{PZ}.

We remark that in case the metric $\chi$ is balanced, then the proof of Theorem \ref{Theorem2} becomes considerably simpler, and indeed does not require the third term of (\ref{TestFunction}). This is made especially clear by, for example, noting the presence of the torsion 1-form in first term on the right-hand side of (\ref{CorollaryAubinYau}).

Additionally we remark that one may obtain another $C^2$ estimate, which is of a different form than (\ref{C2Estimate}), by using the term $e^{A(\sup\varphi - \varphi)}$ from \cite{G} rather than the terms $-A\varphi + \frac{1}{\varphi - \inf\varphi + 1}$ in (\ref{TestFunction}).

This paper is structured as follows. In section \ref{SectionSetupAndNotation} we define all the objects that appear in the paper and review their basic properties. In section \ref{SectionPreliminaryEstimates} we prove some fundamental estimates that we will need many times. In section \ref{SectionEvolutionEquations} we compute the evolution of the test function $G$ along the flow. In section \ref{SectionAubinYauInequality} we prove Theorem \ref{Theorem3}. In section \ref{SectionSecondOrderEstimate} we prove Theorem \ref{Theorem2} and hence Theorem \ref{Theorem1}. \\

\noindent {\bf Acknowledgments:} I would like to thank my advisor D. H. Phong for constant, ongoing encouragement and support.

\section{Setup and notation}\label{SectionSetupAndNotation}

In this section we fix the notation that will be used throughout this paper.

Let $(X,\chi)$ be an $n$-dimensional compact Hermitian manifold. In a local holomorphic coordinate system we write $\chi = i\chi_{\bar kj}dz^j\wedge d\bar z^k$ and $\chi^{j\bar k}\chi_{\bar k\ell} = \delta^j{}_\ell$.

Corresponding to the Hermitian metric $\chi$ is the Chern connection, which is given on sections of $T^{1,0}X$ by
\be
\nabla_jV^p = \partial_jV^p + \Gamma^p_{jq}V^q, \ \ \ \ \ \nabla_{\bar k}V^j = \partial_{\bar k}V^p + \Gamma^p_{\bar kq}V^q
\ee
where
\be
\Gamma^p_{jq} = \chi^{p\bar r}\partial_j\chi_{\bar rq}, \ \ \ \ \ \Gamma^p_{\bar kq} = 0.
\ee
The Chern connection has torsion and torsion 1-form
\be
T^p_{jq} = \Gamma^p_{jq} - \Gamma^p_{qj}, \ \ \ \ \ T_j = T^p_{pj}
\ee
which are both generally non-zero. It also has curvature
\be
R_{\bar kj}{}^p{}_q = -\partial_{\bar k}\Gamma^p_{jq}
\ee
which satisfies the defining property that
\be\label{commutation_relation}
[\nabla_j,\nabla_{\bar k}]V^p = R_{\bar kj}{}^p{}_qV^q.
\ee
The metric defines a Laplacian
\be
\Delta = \chi^{j\bar k}\nabla_j\nabla_{\bar k}.
\ee
Now we turn our attention to equation (\ref{ParabolicComplexMongeAmpereEquation}). Since $F' > 0$, the flow is parabolic and exists for a short time. The remainder of this paper is dedicated to proving a priori estimates, so from now on we assume that $u$ solves (\ref{ParabolicComplexMongeAmpereEquation}) on an interval $[0,T]$, where $0 < T \leq \infty$, with $u(0) = u_0$, as well as the condition $\chi + i\ddb u_0 > 0$. We set
\be
\varphi = u - \frac{1}{V}\int_Xu\,\chi^n
\ee
where $V = \int_X\chi^n$.

Along the flow, we define a $(1,1)$-form
\be
g_{\bar kj} = \chi_{\bar kj} + u_{\bar kj}
\ee
as well as an endomorphism of $T^{1,0}X$
\be
h^p{}_q = \chi^{p\bar r}g_{\bar rq}.
\ee
The determinant estimate in the next section, together with the assumption that $\chi + i\ddb u_0 > 0$, implies that $g$ and $h$ remain positive along the flow. Therefor we may also consider the inverse matrix $g^{j\bar k}g_{\bar k\ell} = \delta^j{}_\ell$ and the inverse endomorphism $h^{-1}$. It follows that $\Tr h = n + \chi^{j\bar k}u_{\bar kj}$ and $\Tr h^{-1} = n - g^{j\bar k}u_{\bar kj}$.

We denote the linearized operator by
\be
L = F' e^{-f}\det h\,g^{j\bar k}\partial_j\partial_{\bar k}.
\ee
Here and throughout we write $F = F(e^{-f}\det h)$, $F' = F'(e^{-f}\det h)$, and so on.

\section{Preliminary estimates}\label{SectionPreliminaryEstimates}

In this section we prove fundamental estimates for the determinant and traces that will be needed many times.

We start by estimating the determinant. It will be convenient to consider the quantity $H = e^{-f}\det h$.

\begin{lemma}\label{LemmaDeterminantEstimate}
With the hypotheses of Theorem \ref{Theorem2},
\be\label{DeterminantEstimate}
C^{-1}\leq H\leq C, \ \ \ \ \ C^{-1}\leq \det h\leq C, \ \ \ \ \ |\partial_tu|\leq C, \ \ \ \ \ |\partial_t\varphi|\leq C.
\ee
In the first two estimates $C = C(\chi,u_0,f)$, and in the last two estimates $C = C(\chi,u_0,f,F)$.
\end{lemma}

\begin{proof}
Since
\be
(\partial_t - L)H = F''Hg^{j\bar k}\partial_jH\partial_{\bar k}H
\ee
we have that $H_{max}$ is decreasing and $H_{min}$ is increasing. The first two estimates follow. The equation $\partial_tu = F$ now gives the last two estimates.
\end{proof}

\noindent Now we prove positive lower bounds for the traces.

\begin{lemma}\label{LemmaTracePositiveLowerBounds}
With the hypotheses of Theorem \ref{Theorem2},
\be
\Tr h \geq C^{-1} \textrm{ and } \Tr h^{-1} \geq C^{-1}.
\ee
Here $C = C(\chi,u_0,f)$.
\end{lemma}

\begin{proof}
The determinant estimate (\ref{DeterminantEstimate}) yields
\be
\Tr h \geq \lambda_{max} \geq \sqrt[n]{\lambda_1\cdots\lambda_n} \geq C^{-1}
\ee
and
\be
\Tr h^{-1} \geq \frac{1}{\lambda_{min}} \geq \frac{1}{\sqrt[n]{\lambda_1\cdots\lambda_n}} \geq C^{-1}.
\ee
\end{proof}

\section{Evolution equations}\label{SectionEvolutionEquations}

In this section we compute $(\partial_t - L)G$.

It is straightforward to check that the following formulas from \cite{PZ} still hold without the assumption that $\chi$ is K\"ahler:
\be\label{evolutionLogTrh1}
(\partial_t - L)\log\Tr h = \frac{(\partial_t - L)\Tr h}{\Tr h} + F'e^{-f}\det h\frac{g^{j\bar k}\partial_j\Tr h\partial_{\bar k}\Tr h}{(\Tr h)^2}
\ee
\begin{align}
(\partial_t - L)\Tr h & = F'e^{-f}\det h\bigg\{\frac{\Delta\det h}{\det h} - \frac{2}{\det h}\Re\langle\partial f,\partial\det h\rangle + e^f\Delta e^{-f} \label{EvolutionTrh} \\
                      & \ \ \ \ \ \ \ \ \ \ \ \ \ \ \ \ \ \ \ \ + \frac{F''}{F'e^{-f}\det h}|\partial(e^{-f}\det h)|^2 - g^{j\bar k}\partial_j\partial_{\bar k}\Delta u\bigg\}. \nonumber
\end{align}
\be\label{EvolutionDeth}
\frac{\Delta\det h}{\det h} = \chi^{p\bar q}g^{j\bar k}\nabla_p\nabla_{\bar q}g_{\bar kj} - \chi^{p\bar q}g^{j\bar k}g^{r\bar s}\nabla_pg_{\bar kr}\nabla_{\bar q}g_{\bar sj} + \frac{|\partial\det h|^2}{(\det h)^2}.
\ee
To proceed we must commute the covariant derivatives in the first term of (\ref{EvolutionDeth}). Because in the non-K\"ahler setting it is no longer true that $[\nabla_j,\nabla_k] = 0$, we will obtain terms involving both the curvature and the torsion of $\chi$. The relevant commutation relations, in the order we will use them, are
\be\label{CommutationRelation1}
[\nabla_{\bar q},\nabla_{\bar k}]u_j = -\partial_{\bar s}u_j\overline{T^s_{qk}}
\ee
\be\label{CommutationRelation2}
[\nabla_p,\nabla_{\bar k}]u_{\bar qj} = u_{\bar sj}R_{\bar kp\bar q}{}^{\bar s} - u_{\bar qr}R_{\bar kp}{}^r{}_j
\ee
\be\label{CommutationRelation3}
[\nabla_p,\nabla_j]u_{\bar q} = -\partial_ru_{\bar q}T^r_{pj}.
\ee
Applying (\ref{CommutationRelation1}), (\ref{CommutationRelation2}), (\ref{CommutationRelation3}) to the first term of (\ref{EvolutionDeth}) gives
\begin{align}
\chi^{p\bar q}g^{j\bar k}\nabla_p\nabla_{\bar q}g_{\bar kj} & = g^{j\bar k}\partial_j\partial_{\bar k}\Delta u + g^{j\bar k}u_{\bar sj}R_{\bar kp}{}^{p\bar s} - g^{j\bar k}u_{\bar qr}R_{\bar k}{}^{\bar qr}{}_j \\
     & \ \ \ \ \ \ \ \ \ \ - \chi^{p\bar q}g^{j\bar k}\nabla_p(u_{\bar sj}\overline{T^s_{qk}}) - \chi^{p\bar q}g^{j\bar k}\nabla_{\bar k}(u_{\bar qr}T^r_{pj}). \nonumber
\end{align}
Using $u_{\bar sr} = g_{\bar sr} - \chi_{\bar sr}$ on the curvature terms gives
\begin{align}
\chi^{p\bar q}g^{j\bar k}\nabla_p\nabla_{\bar q}g_{\bar kj} & = g^{j\bar k}\partial_j\partial_{\bar k}\Delta u + R' - h^{-1j}{}_kR^k{}_p{}^q{}_jh^p{}_q \label{FourthOrderTerm} \\
     & \ \ \ \ \ \ \ \ \ \ - \chi^{p\bar q}g^{j\bar k}\nabla_p(u_{\bar sj}\overline{T^s_{qk}}) - \chi^{p\bar q}g^{j\bar k}\nabla_{\bar k}(u_{\bar qr}T^r_{pj}) \nonumber
\end{align}
where $R' = R^j{}_p{}^p{}_j$.

Substituting (\ref{FourthOrderTerm}), (\ref{EvolutionDeth}), (\ref{EvolutionTrh}) into (\ref{evolutionLogTrh1}), and expanding the $\partial(e^{-f}\det h)$ term of (\ref{EvolutionTrh}) as well as the torsion terms of (\ref{FourthOrderTerm}) and using $\nabla\chi = 0$, we find
\begin{align}
  (\partial_t - L)\log\Tr & h = \label{EvolutionLogTrh2} \\
 \frac{F'e^{-f}\det h}{\Tr h}&\bigg\{\frac{g^{j\bar k}\partial_j\Tr h\partial_{\bar k}\Tr h}{\Tr h} - \chi^{p\bar q}g^{j\bar k}g^{r\bar s}\nabla_pg_{\bar kr}\nabla_{\bar q}g_{\bar sj} \nonumber \\
                             & - \chi^{p\bar q}g^{j\bar k}\nabla_pg_{\bar sj}\overline{T^s_{qk}} - \chi^{p\bar q}g^{j\bar k}\nabla_{\bar k}g_{\bar qr}T^r_{pj} - \chi^{p\bar q}g^{j\bar k}u_{\bar sj}\nabla_p\overline{T^s_{qk}} - \chi^{p\bar q}g^{j\bar k}u_{\bar qr}\nabla_{\bar k}T^r_{pj} \nonumber \\
                             & + \left(\frac{1}{(\det h)^2} + \frac{F''e^{-f}}{F'\det h}\right)|\partial\det h|^2 - \left(\frac{1}{\det h} + \frac{F''e^{-f}}{F'}\right)2\Re\langle\partial\det h,\partial f\rangle \nonumber \\
  			                     & + \frac{F''e^{-f}\det h}{F'}|\partial f|^2 + e^f\Delta e^{-f} + R' - h^{-1j}{}_kR^k{}_p{}^q{}_jh^p{}_q \bigg\}. \nonumber
\end{align}
Now we apply $(\partial_t - L)$ to the other terms of (\ref{TestFunction}).
\be\label{EvolutionPhi}
-A(\partial_t - L)\varphi = -A\partial_t\varphi + AnF'e^{-f}\det h - AF'e^{-f}\det h\Tr h^{-1}
\ee
\be\label{EvolutionPhi-1}\begin{split}
(\partial_t - L)\frac{1}{\varphi - \inf\varphi + 1} & = -\frac{1}{(\varphi - \inf\varphi + 1)^2}\partial_t\varphi - \frac{2F'e^{-f}\det h}{(\varphi - \inf\varphi + 1)^3}g^{j\bar k}\partial_j\varphi\partial_{\bar k}\varphi \\ & \ \ \ \ + \frac{nF'e^{-f}\det h}{(\varphi - \inf\varphi + 1)^2} - \frac{F'e^{-f}\det h}{(\varphi - \inf\varphi + 1)^2}\Tr h^{-1}
\end{split}\ee
\be\label{EvolutionF}
\frac{B}{2}(\partial_t - L)F^2 = -BF'e^{-f}\det hg^{j\bar k}\partial_jF\partial_{\bar k}F.
\ee
Combining (\ref{EvolutionLogTrh2}), (\ref{EvolutionPhi}), (\ref{EvolutionPhi-1}), (\ref{EvolutionF}) we find
\begin{align}
(\partial_t & -  L) G = \label{EvolutionG} \\
& \frac{F'e^{-f}\det h}{\Tr h}\bigg\{\frac{g^{j\bar k}\partial_j\Tr h\partial_{\bar k}\Tr h}{\Tr h} - \chi^{p\bar q}g^{j\bar k}g^{r\bar s}\nabla_pg_{\bar kr}\nabla_{\bar q}g_{\bar sj} \label{EvolutionG1} \\
& \ \ \ \ \ \ \ - \chi^{p\bar q}g^{j\bar k}\nabla_pg_{\bar sj}\overline{T^s_{qk}} - \chi^{p\bar q}g^{j\bar k}\nabla_{\bar k}g_{\bar qr}T^r_{pj} - \chi^{p\bar q}g^{j\bar k}u_{\bar sj}\nabla_p\overline{T^s_{qk}} - \chi^{p\bar q}g^{j\bar k}u_{\bar qr}\nabla_{\bar k}T^r_{pj} \label{EvolutionG2} \\
& \ \ \ \ \ \ \ + \left(\frac{1}{(\det h)^2} + \frac{F''e^{-f}}{F'\det h}\right)|\partial\det h|^2 - \left(\frac{1}{\det h} + \frac{F''e^{-f}}{F'}\right)2\Re\langle\partial\det h,\partial f\rangle \label{EvolutionG3} \\
& \ \ \ \ \ \ \ + \frac{F''e^{-f}\det h}{F'}|\partial f|^2 + e^f\Delta e^{-f} + R' - h^{-1j}{}_kR^k{}_p{}^q{}_jh^p{}_q \bigg\} \label{EvolutionG4} \\
& - A\partial_t\varphi + AnF'e^{-f}\det h - AF'e^{-f}\det h\Tr h^{-1}  \label{EvolutionG5} \\
& - \frac{1}{(\varphi - \inf\varphi + 1)^2}\partial_t\varphi - \frac{2F'e^{-f}\det h}{(\varphi - \inf\varphi + 1)^3}g^{j\bar k}\partial_j\varphi\partial_{\bar k}\varphi \label{EvolutionG6} \\
& + \frac{nF'e^{-f}\det h}{(\varphi - \inf\varphi + 1)^2} - \frac{F'e^{-f}\det h}{(\varphi - \inf\varphi + 1)^2}\Tr h^{-1} \label{EvolutionG7} \\
& - BF'e^{-f}\det hg^{j\bar k}\partial_jF\partial_{\bar k}F. \label{EvolutionG8}
\end{align}

\noindent We note that though some terms in this expression are complex, nonetheless their sum is real. Hence in what follows we need only to estimate the real part of each term.

We now discuss the role of the many terms in (\ref{EvolutionG}).

The first term of (\ref{EvolutionG1}) is the primary bad term. The second term of (\ref{EvolutionG1}) is the primary good term. The terms of (\ref{EvolutionG2}) are new torsion terms that do not appear in the K\"ahler case.

In the K\"ahler case, the primary good term is able to defeat the primary bad term. In fact, this is still true in the balanced case (see the remark in the introduction). However it is no longer true in the Hermitian case.

Fortunately, though the primary good term is only able to defeat a part of the primary bad term, it is nonetheless able to defeat the new torsion terms at the same time. The special coordinate system (\ref{GLCoordinates}) below facilitates this.

Bounding the remaining part of the primary bad term requires help from the terms of (\ref{EvolutionG6}), (\ref{EvolutionG7}), and (\ref{EvolutionG8}). This is the reason for including $\frac{1}{\varphi - \inf\varphi + 1}$ in (\ref{TestFunction}).

The first term of (\ref{EvolutionG3}) is another bad term. It is defeated by the term of (\ref{EvolutionG8}). This is the purpose of the $F^2$ term in (\ref{TestFunction}).

\section{Aubin-Yau inequality}\label{SectionAubinYauInequality}

In this section we work only with two arbitrary Hermitian metrics $\chi$, $\omega$ on a compact complex manifold, and prove Theorem \ref{Theorem3}. As a result, we will have obtained control over the terms of (\ref{EvolutionG1}) and (\ref{EvolutionG2}). Recall that we define $\alpha_{\bar qpj} = \partial_ju_{\bar qp} - \partial_pu_{\bar qj}$ and $\alpha_j = \chi^{p\bar q}\alpha_{\bar qpj}$, where $u_{\bar kj} = g_{\bar kj} - \chi_{\bar kj}$.

We work at a point $x_0$ using the coordinate system of \cite{GL}, which satisfies
\be\label{GLCoordinates}
\chi_{\bar kj}(x_0) = \delta_{\bar kj}, \ \ \ \ \ \partial_j\chi_{\bar kk}(x_0) = 0, \ \ \ \ \ g_{\bar kj}(x_0) \textrm{ is diagonal}.
\ee
We first compute the bad term. In this coordinate system we have $\partial_j\Tr h = \sum_q\partial_ju_{\bar qq}$.
\begin{align}\label{BadTerm}
g^{j\bar k}\partial_j\Tr h\partial_{\bar k}\Tr h
 & = \sum g^{k\bar k}\partial_ku_{\bar qq}\partial_{\bar k}u_{\bar ss} \\
 & = \sum g^{k\bar k}(\partial_qu_{\bar qk} + \alpha_{\bar qqk})(\partial_{\bar s}u_{\bar ks} + \overline{\alpha_{\bar ssk}}) \nonumber \\
 & = \sum_{k,q,s}g^{k\bar k}\partial_qu_{\bar qk}\partial_{\bar s}u_{\bar ks} + 2\Re\chi^{p\bar q}g^{j\bar k}\partial_pu_{\bar qj}\overline{\alpha_k} + g^{j\bar k}\alpha_j\overline{\alpha_k} \nonumber \\
 & = \chi^{p\bar q}\chi^{r\bar s}g^{j\bar k}\partial_pg_{\bar qj}\partial_{\bar s}g_{\bar kr} - \chi^{p\bar q}\chi^{r\bar s}g^{j\bar k}(\partial_pg_{\bar qj}\partial_{\bar s}\chi_{\bar kr} + \partial_p\chi_{\bar qj}\partial_{\bar s}g_{\bar kr} - \partial_p\chi_{\bar qj}\partial_{\bar s}\chi_{\bar kr}) \nonumber \\
 & \ \ \ \ \ \ \ \ \ \ \ \ \ \ \ \ \ \ \ \ \ \ \ \ \ \ \ \ \ + 2\Re\chi^{p\bar q}g^{j\bar k}\partial_pu_{\bar qj}\overline{\alpha_k} + g^{j\bar k}\alpha_j\overline{\alpha_k} \nonumber \\
 & = \chi^{p\bar q}\chi^{r\bar s}g^{j\bar k}\partial_pg_{\bar qj}\partial_{\bar s}g_{\bar kr} - \chi^{p\bar q}\chi^{r\bar s}g^{j\bar k}(\partial_pu_{\bar qj}\partial_{\bar s}\chi_{\bar kr} + \partial_p\chi_{\bar qj}\partial_{\bar s}u_{\bar kr} + \partial_p\chi_{\bar qj}\partial_{\bar s}\chi_{\bar kr}) \nonumber \\
 & \ \ \ \ \ \ \ \ \ \ \ \ \ \ \ \ \ \ \ \ \ \ \ \ \ \ \ \ \ + 2\Re\chi^{p\bar q}g^{j\bar k}\partial_pu_{\bar qj}\overline{\alpha_k} + g^{j\bar k}\alpha_j\overline{\alpha_k} \nonumber \\
 & = \chi^{p\bar q}\chi^{r\bar s}g^{j\bar k}\partial_pg_{\bar qj}\partial_{\bar s}g_{\bar kr} - 2\Re\chi^{p\bar q}g^{j\bar k}\partial_pu_{\bar qj}\overline{T_k}  - g^{j\bar k}T_j\overline{T_k} \nonumber \\
 & \ \ \ \ \ \ \ \ \ \ \ \ \ \ \ \ \ \ \ \ \ \ \ \ \ \ \ \ \ + 2\Re\chi^{p\bar q}g^{j\bar k}\partial_pu_{\bar qj}\overline{\alpha_k} + g^{j\bar k}\alpha_j\overline{\alpha_k} \nonumber \\
 & = \chi^{p\bar q}\chi^{r\bar s}g^{j\bar k}\partial_pg_{\bar qj}\partial_{\bar s}g_{\bar kr} - 2\Re\chi^{p\bar q}g^{j\bar k}(\partial_ju_{\bar qp} - \alpha_{\bar qpj})\overline{T_k}  - g^{j\bar k}T_j\overline{T_k} \nonumber \\
 & \ \ \ \ \ \ \ \ \ \ \ \ \ \ \ \ \ \ \ \ \ \ \ \ \ \ \ \ \ + 2\Re\chi^{p\bar q}g^{j\bar k}(\partial_ju_{\bar qp} - \alpha_{\bar qpj})\overline{\alpha_k} + g^{j\bar k}\alpha_j\overline{\alpha_k} \nonumber \\
 & = \chi^{p\bar q}\chi^{r\bar s}g^{j\bar k}\partial_pg_{\bar qj}\partial_{\bar s}g_{\bar kr} - 2\Re g^{j\bar k}\partial_j\Tr h\overline{T_k} + 2\Re g^{j\bar k}\alpha_j\overline{T_k}  - g^{j\bar k}T_j\overline{T_k} \nonumber \\
 & \ \ \ \ \ \ \ \ \ \ \ \ \ \ \ \ \ \ \ \ \ \ \ \ \ \ \ \ \ + 2\Re g^{j\bar k}\partial_j\Tr h\overline{\alpha_k} - g^{j\bar k}\alpha_j\overline{\alpha_k}. \nonumber
\end{align}
The good term is
\be\label{GoodTerm}
-\chi^{p\bar q}g^{j\bar k}g^{r\bar s}\nabla_pg_{\bar sj}\nabla_{\bar q}g_{\bar kr} = -\chi^{p\bar q}g^{j\bar k}g^{r\bar s}\partial_pg_{\bar sj}\partial_{\bar q}g_{\bar kr} + 2\Re\chi^{p\bar q}g^{j\bar k}\partial_pg_{\bar sj}\overline{\Gamma^s_{qk}} - \chi^{p\bar q}g^{j\bar k}g_{\bar sr}\Gamma^r_{pj}\overline{\Gamma^s_{qk}}.
\ee
Now we have
\begin{align}\label{TorsionTerms}
    -\chi^{p\bar q}g^{j\bar k}\nabla_pu_{\bar sj}\overline{T^s_{qk}} - \chi^{p\bar q}g^{j\bar k}\nabla_{\bar k}u_{\bar qr}T^r_{pj}
= & -2\Re\chi^{p\bar q}g^{j\bar k}\partial_pg_{\bar sj}\overline{T^s_{qk}} \\
  & + \chi^{p\bar q}g^{j\bar k}g_{\bar sr}\Gamma^r_{pj}\overline{\Gamma^s_{qk}} - \chi^{p\bar q}g^{j\bar k}g_{\bar sr}\Gamma^r_{jp}\overline{\Gamma^s_{kq}} \nonumber \\
  & + \chi^{p\bar q}g^{j\bar k}\chi_{\bar sr}T^r_{pj}\overline{T^s_{qk}} - \chi^{p\bar q}g^{j\bar k}T^r_{pj}\overline{\alpha_{\bar rqk}}. \nonumber
\end{align}
Note that the second term on the right-hand side of (\ref{GoodTerm}) may combine with the first term on the right-hand side of (\ref{TorsionTerms}), using $\Gamma^s_{qk} - T^s_{qk} = \Gamma^s_{kq}$. Also note that the third term on the right-hand side of (\ref{GoodTerm}) cancels with the second term on the right-hand side of (\ref{TorsionTerms}). Thus we see that
\begin{align}
& \frac{g^{j\bar k}\partial_j\Tr h\partial_{\bar k}\Tr h}{\Tr h} - \chi^{p\bar q}g^{j\bar k}g^{r\bar s}\nabla_pg_{\bar sj}\nabla_{\bar q}g_{\bar kr} - \Re\chi^{p\bar q}g^{j\bar k}\nabla_pu_{\bar sj}\overline{T^s_{qk}} - \Re\chi^{p\bar q}g^{j\bar k}\nabla_{\bar k}u_{\bar qr}T^r_{pj} \nonumber \\
= \ & \frac{1}{\Tr h}\chi^{p\bar q}\chi^{r\bar s}g^{j\bar k}\partial_pg_{\bar qj}\partial_{\bar s}g_{\bar kr} - 2\Re g^{j\bar k}\frac{\partial_j\Tr h}{\Tr h}\overline{T_k} - \frac{1}{\Tr h}g^{j\bar k}T_j\overline{T_k} \nonumber \\
    & + 2\Re g^{j\bar k}\frac{\partial_j\Tr h}{\Tr h}\overline{\alpha_k} + \frac{1}{\Tr h}2\Re g^{j\bar k}\alpha_j\overline{T_k} - \frac{1}{\Tr h}g^{j\bar k}\alpha_j\overline{\alpha_k} \nonumber \\
    & - \chi^{p\bar q}g^{j\bar k}g^{r\bar s}\partial_pg_{\bar sj}\partial_{\bar q}g_{\bar kr} + 2\Re\chi^{p\bar q}g^{j\bar k}\partial_pg_{\bar sj}\overline{\Gamma^s_{kq}} \nonumber \\
		& - \chi^{p\bar q}g^{j\bar k}g_{\bar sr}\Gamma^r_{jp}\overline{\Gamma^s_{kq}} + \chi^{p\bar q}g^{j\bar k}\chi_{\bar sr}T^r_{pj}\overline{T^s_{qk}} - \Re\chi^{p\bar q}g^{j\bar k}T^r_{pj}\overline{\alpha_{\bar rqk}}. \label{AubinYauIdentity}
\end{align}
The first term on the right-hand side of (\ref{AubinYauIdentity}) may be estimated as follows, by a classic technique using the Cauchy-Schwarz inequality twice:
\begin{align}
\chi^{p\bar q}\chi^{r\bar s}g^{j\bar k}\partial_pg_{\bar qj}\partial_{\bar s}g_{\bar kr}
 & = \sum_{s,q}\sum_kg^{k\bar k}\partial_sg_{\bar sk}\partial_{\bar q}g_{\bar kq} \label{AubinYauEstimate1} \\
 & \leq \sum_{s,q}\sqrt{\sum_kg^{k\bar k}|\partial_sg_{\bar sk}|^2}\sqrt{\sum_kg^{k\bar k}|\partial_qg_{\bar qk}|^2} \nonumber \\
 & = \left(\sum_s\sqrt{\sum_kg^{k\bar k}|\partial_sg_{\bar sk}|^2}\right)^2 \nonumber \\
 & = \left(\sum_s\sqrt{g_{\bar ss}}\sqrt{\sum_kg^{s\bar s}g^{k\bar k}|\partial_sg_{\bar sk}|^2}\right)^2 \nonumber \\
 & \leq (\Tr h)\cdot\sum_{k,s}g^{k\bar k}g^{s\bar s}|\partial_sg_{\bar sk}|^2. \nonumber
\end{align}
Also the eighth term on the right-hand side of (\ref{AubinYauIdentity}) may be broken apart, using the special coordinate system, as
\be\label{AubinYauEstimate2}
2\Re\chi^{p\bar q}g^{j\bar k}\partial_pg_{\bar sj}\overline{\Gamma^s_{kq}} \leq \sum_{q\neq s}g^{k\bar k}g^{s\bar s}|\partial_qg_{\bar sk}|^2 + \chi^{p\bar q}g^{j\bar k}g_{\bar sr}\Gamma^r_{jp}\overline{\Gamma^s_{kq}}.
\ee
Hence from (\ref{AubinYauIdentity}), (\ref{AubinYauEstimate1}), (\ref{AubinYauEstimate2}) we obtain
\be\begin{split}
& \frac{g^{j\bar k}\partial_j\Tr h\partial_{\bar k}\Tr h}{\Tr h} - \chi^{p\bar q}g^{j\bar k}g^{r\bar s}\nabla_pg_{\bar sj}\nabla_{\bar q}g_{\bar kr} - \Re\chi^{p\bar q}g^{j\bar k}\nabla_pu_{\bar sj}\overline{T^s_{qk}} - \Re\chi^{p\bar q}g^{j\bar k}\nabla_{\bar k}u_{\bar qr}T^r_{pj} \\
\leq \ & - 2\Re g^{j\bar k}\frac{\partial_j\Tr h}{\Tr h}\overline{T_k} - \frac{1}{\Tr h}g^{j\bar k}T_j\overline{T_k} + 2\Re g^{j\bar k}\frac{\partial_j\Tr h}{\Tr h}\overline{\alpha_k} + \frac{1}{\Tr h}2\Re g^{j\bar k}\alpha_j\overline{T_k} - \frac{1}{\Tr h}g^{j\bar k}\alpha_j\overline{\alpha_k} \\
		& + \chi^{p\bar q}g^{j\bar k}\chi_{\bar sr}T^r_{pj}\overline{T^s_{qk}} - \Re\chi^{p\bar q}g^{j\bar k}T^r_{pj}\overline{\alpha_{\bar rqk}}.
\end{split}\ee
This proves Theorem \ref{Theorem3}.

\ActivateWarningFilters[pdflatex]
\section{$C^2$ estimate}\label{SectionSecondOrderEstimate}
\DeactivateWarningFilters[pdflatex]

Now we prove Theorem \ref{Theorem2}. We work at a maximum point $(x_0,t_0)$ of $G$. Since $\partial_jG = 0$ at this point, we have
\be\label{MaxPointEquality}
\frac{\partial_j\Tr h}{\Tr h} = A\partial_j\varphi + \frac{1}{(\varphi - \inf\varphi + 1)^2}\partial_j\varphi - BF\partial_jF.
\ee
Recalling the expression (\ref{EvolutionG}) for $(\partial_t - L)G$ and applying Corollary \ref{Corollary1} shows that we now need only to estimate the term involving $g^{j\bar k}\partial_j\Tr h\overline{T_k}$ from (\ref{CorollaryAubinYau}) and the terms controlled by $|\partial\det h|^2$ from (\ref{EvolutionG}).

Using (\ref{MaxPointEquality}) and incorporating the factor of $F'e^{-f}\det h/\Tr h$ from outside the brackets in (\ref{EvolutionG}), we have that the first term coming from the right-hand side of (\ref{CorollaryAubinYau}) may be estimated in the following way, as long as $A > 1$:
\be\label{BadTerm4}\begin{split}
& -\frac{F'e^{-f}\det h}{\Tr h} 2\Re g^{j\bar k}\frac{\partial_j\Tr h}{\Tr h}\overline{T_k} \\
& \ \ \ \ \ \ \ \ = -\frac{F'e^{-f}\det h}{\Tr h}2\Re g^{j\bar k}\left(A\partial_j\varphi + \frac{1}{(\varphi - \inf\varphi + 1)^2}\partial_j\varphi - BF\partial_jF\right)\overline{T_k} \\
& \ \ \ \ \ \ \ \ \leq \frac{F'e^{-f}\det h}{(\varphi - \inf\varphi + 1)^3}g^{j\bar k}\partial_j\varphi\partial_{\bar k}\varphi + \frac{(\varphi - \inf\varphi + 1)^3}{(\Tr h)^2}CA^2\Tr h^{-1} \\
& \ \ \ \ \ \ \ \ \ \ \ \ + \frac{F'e^{-f}\det h}{(\varphi - \inf\varphi + 1)^3}g^{j\bar k}\partial_j\varphi\partial_{\bar k}\varphi + \frac{(\varphi - \inf\varphi + 1)^{-1}}{(\Tr h)^2}C\Tr h^{-1} \\
& \ \ \ \ \ \ \ \ \ \ \ \ \ \ \ \ + F'e^{-f}\det hg^{j\bar k}\partial_jF\partial_{\bar k}F + CB^2\Tr h^{-1} \\
& \ \ \ \ \ \ \ \ \leq \frac{2F'e^{-f}\det h}{(\varphi - \inf\varphi + 1)^3}g^{j\bar k}\partial_j\varphi\partial_{\bar k}\varphi + \frac{(\varphi - \inf\varphi + 1)^3}{(\Tr h)^2}CA^2\Tr h^{-1} \\
& \ \ \ \ \ \ \ \ \ \ \ \ \ \ \ \ + F'e^{-f}\det hg^{j\bar k}\partial_jF\partial_{\bar k}F + CB^2\Tr h^{-1}.
\end{split}\ee
Adding the terms from (\ref{EvolutionG5}), (\ref{EvolutionG6}), (\ref{EvolutionG7}), (\ref{EvolutionG8}) to (\ref{BadTerm4}) and using a positive lower bound for $F'e^{-f}\det h$ gives
\be\label{SurvivingTerms}\begin{split}
-\frac{F'e^{-f}\det h}{\Tr h}2\Re g^{j\bar k}&\frac{\partial_j\Tr h}{\Tr h}\overline{T^q_{qk}} + (\partial_t - L)\left\{-A\varphi + \frac{1}{\varphi - \inf\varphi + 1} + \frac{B}{2}F^2\right\} \\
& \leq CA + (CB^2 - \epsilon A)\Tr h^{-1} + \frac{(\varphi - \inf\varphi + 1)^3}{(\Tr h)^2}CA^2\Tr h^{-1} \\
& \ \ \ \ \ \ \ \ \ \ + (1-B)F'e^{-f}\det hg^{j\bar k}\partial_jF\partial_{\bar k}F.
\end{split}\ee
Now as long as $B > 1$ then
\begin{align}
(1-B)F'e^{-f}\det hg^{j\bar k}\partial_jF\partial_{\bar k}F
& = \phantom{+} (1-B)(F')^3 e^{-3f}  \det h         g^{j\bar k}\partial_j \det h \partial_{\bar k} \det h \nonumber \\
& \phantom{=} + (1-B)(F')^3 e^{-3f} (\det h)^3      g^{j\bar k}\partial_j f      \partial_{\bar k} f      \nonumber \\
& \phantom{=} - (1-B)(F')^3 e^{-3f} (\det h)^2 2\Re g^{j\bar k}\partial_j \det h \partial_{\bar k} f      \nonumber \\
& \leq (2 - \epsilon B) g^{j\bar k}\partial_j\det h\partial_{\bar k}\det h + CB^2\Tr h^{-1}. \nonumber
\end{align}
Since
\be
\frac{1}{\Tr h}|\partial\det h|^2 \leq g^{j\bar k}\partial_j\det h\partial_{\bar k}\det h
\ee
we see that for $B \gg 1$ the term $(1-B)F'e^{-f}\det hg^{j\bar k}\partial_jF\partial_{\bar k}F$ overtakes the final unwanted terms from (\ref{EvolutionG3}).

Putting all of the above estimates together, it follows that at $(x_0,t_0)$
\be\begin{split}
0 & \leq (\partial_t - L)G \\
  & \leq CA + (CB^2 - \epsilon A)\Tr h^{-1} + \frac{(\varphi - \inf\varphi + 1)^3}{(\Tr h)^2}CA^2\Tr h^{-1}.
\end{split}\ee
Take $A \gg B$ so that
\be\label{Trh-1EstimateAtMaxPoint}
0 \leq CA - 2\Tr h^{-1} + \frac{(\varphi - \inf\varphi + 1)^3}{(\Tr h)^2}CA^2\Tr h^{-1}.
\ee
Now we consider two cases. If
\be\label{Case1}
(\Tr h)^2 \leq CA^2(\varphi - \inf\varphi + 1)^3
\ee
then since $\frac{3}{2}\log s \leq s$ we have
\begin{align}
G & \leq G(x_0,t_0) \\
  & \leq \log\Tr h(x_0,t_0) - A\varphi(x_0,t_0) + 1 + \frac{B}{2}F(x_0,t_0)^2 \nonumber \\
	& \leq \frac{1}{2}\log(CA^2(\varphi(x_0,t_0) - \inf\varphi + 1)^3) - A\varphi(x_0,t_0) + 1 + \frac{B}{2}F(x_0,t_0)^2 \nonumber \\
	& \leq \varphi(x_0,t_0) - \inf\varphi + 1 + \frac{1}{2}\log (CA^2) - A\varphi(x_0,t_0) + 1 + \frac{B}{2}F(x_0,t_0)^2 \nonumber \\
	& = -(A-1)\varphi(x_0,t_0) - \inf\varphi + 2 + \frac{1}{2}\log (CA^2) + \frac{B}{2}F(x_0,t_0)^2 \nonumber \\
	& \leq -A\inf\varphi + 2 + \frac{1}{2}\log (CA^2) + \frac{B}{2}F(x_0,t_0)^2. \nonumber
\end{align}
Since $F$ is uniformly bounded this implies (\ref{C2Estimate}) directly.

The alternative to (\ref{Case1}) is
\be\label{Case2}
\frac{(\varphi - \inf\varphi + 1)^3}{(\Tr h)^2}CA^2 \leq 1.
\ee
In this case (\ref{Trh-1EstimateAtMaxPoint}) implies $\Tr h^{-1}(x_0,t_0)\leq C$. The determinant estimate now gives $\Tr h(x_0,t_0)\leq C$. Hence
\begin{align}
\log\Tr h - A\varphi & \leq G \\
                     & \leq G(x_0,t_0) \nonumber \\
										 & \leq \log\Tr h(x_0,t_0) - A\inf\varphi + 1 + \frac{B}{2}F(x_0,t_0)^2 \nonumber \\
										 & \leq \log\Tr h(x_0,t_0) - A\inf\varphi + C \nonumber
\end{align}
and finally
\be
\Tr h \leq Ce^{A(\varphi - \inf\varphi)}.
\ee
This completes the proof of Theorem \ref{Theorem2} and hence the proof of Theorem \ref{Theorem1}.

        Department of Mathematics, Columbia University, New York, NY 10027, USA \\
\indent kjs@math.columbia.edu

\end{document}